\documentclass[11pt]{amsart}
\usepackage[french, ngerman, italian, english]{babel}
\usepackage{amsmath,amsthm, amssymb, amsfonts}

\usepackage{setspace}
\usepackage{anysize}
\usepackage{url}
\usepackage{bm}
\usepackage{color,graphicx}
\usepackage[colorlinks=false]{hyperref}
\usepackage{paralist}
\usepackage{verbatim}
\usepackage[utf8]{inputenc}
\usepackage[all]{xy}
\usepackage{pdfsync}

\theoremstyle{plain}
\newtheorem{theorem}{\bf Theorem}[section]

\newtheorem*{theorem*}{Theorem}
\newtheorem{proposition}[theorem]{\bf Proposition}
\newtheorem{lemma}[theorem]{\bf Lemma}
\newtheorem{corollary}[theorem]{\bf Corollary}

\newtheorem*{conjecture*}{\bf Conjecture}
\theoremstyle{definition}
\newtheorem{definition}[theorem]{\bf Def\mbox{}inition}
\newtheorem{remark}[theorem]{\bf Remark}
\newtheorem{algorithm}[theorem]{\bf Algorithm}

\theoremstyle{remark}

\theoremstyle{example}

\def \ld{{\operatorname{ld}}}
\def \lin{{\operatorname{lin}}}

\def\ini{\operatorname{\rm in}}

\def \reg{{\operatorname{reg}}}
\def \creg{{\operatorname{creg}}}

\def \ZZ{\mathbb Z}
\newcommand{\lr}[1]{\langle #1 \rangle}

\usepackage{amscd}
\begin{document}
\title{Componentwise regularity (I)}
\author[G. Caviglia]{Giulio Caviglia}  \thanks{The work of the first author was supported by a grant from the
Simons Foundation (209661 to G.~C.). }
\address{Department of Mathematics, Purdue University 150 N. University Street, West Lafayette, IN 47907-2067}
\email{gcavigli@math.purdue.edu}
\author[M. Varbaro]{Matteo Varbaro}
\address{Dipartimento di Matematica,
Universit\`a degli Studi di Genova, Via Dodecaneso 35, 16146, Italy}
\email{varbaro@dima.unige.it}
\maketitle

\begin{abstract} We define the notion of componentwise regularity and study some of its basic properties. We prove an analogue, when working with weight orders, of  Buchberger's criterion  to compute Gr\"obner bases; the proof of our criterion relies on a strengthening of a lifting lemma of Buchsbaum and Eisenbud. This criterion helps us to show a stronger version of Green's crystallization theorem in a quite general setting, according to the componentwise regularity of the initial object. Finally we show a necessary condition, given a submodule $M$ of a free one over the polynomial ring and a weight such that $\ini(M)$ is componentwise linear, for the existence of an $i$ such that $\beta_i(M)=\beta_i(\ini(M))$. 
\end{abstract}
\maketitle
\section{Introduction}
The notion of componentwise linearity was introduced in \cite{HH}, and since then componentwise linear modules have demonstrated to be quite ubiquitous objects \cite{HRW}, \cite{AHH}, \cite{R}, \cite{C}, 
 \cite{CHH}, \cite{IR}, \cite{SV},  \cite{RS},   \cite{NR}, \cite{CS}.
 Despite the fact that componentwise linear modules can be defined over any standard graded ring, the definition works better in a Koszul ambient space. 
 
 Let $R$ be a standard graded $K$-algebra,  and denote with $\mathbf m$ its homogeneous maximal ideal.  
 Given a finitely generated graded $R$-module $M$, we write $\reg_R(M)$ for its Castelnuovo-Mumford regularity and 
 when there is no ambiguity we denote $\reg_R(M)$ simply by $\reg(M).$

\begin{definition}
Let $M=\bigoplus_{d\in \mathbb Z} M_d$ be a finitely generated $\mathbb Z$-graded $R$-module and, for every 
$a\in \ZZ$, denote with $M_{\langle a \rangle}= \bigoplus_{d\geq 0} \mathbf m^d M_a$ the submodule of $M$ generated by $M_a.$ We say that $M$ is \emph{$r$-componentwise regular} if $\reg (M_{\lr a})\leq a+r$ for all $a.$ The \emph{componentwise regularity of  $M$} is the infimum of all $r$ such that $M$ is $r$-componentwise regular; we will denoted it by $\creg(M).$
\end{definition} 

Clearly a graded module $M\not =0$ is componentwise linear if and only if its componentwise regularity equals zero.
When $R$ is Kozsul, i.e. the residue field $K=R/\mathbf m$ has a linear graded free resolution over $R,$
for every $d\geq 0$, the ideal $\mathbf m^d$ has regularity equal to $d;$ in particular $R$, and more generally every direct sum of shifted copies  of $R$,  has componentwise regularity zero. 
When $R$ is Koszul, one sees immediately that  if a module $M$ is generated in a single degree $d$, then it has Castelnuovo-Mumford  regularity $d+r$ if and only if  $\creg(M)=r.$ As we will explain in Section \ref{sec-creg}, eventually, one can see that $R$ is Koszul if and only if every finitely generate graded $R$-module $M$ has finite componentwise regularity, despite the fact that $\creg(M)$ can be much bigger (as well as much smaller) than $\reg(M)$.

The main reason leading us to introduce the concept of componentwise regularity in the present paper, is the feeling that it may provide a better approximation for the complexity of computing Gr\"obner bases than the Castelnuovo-Mumford regularity. In some sense this belief is expressed by the generalization of Green's Crystallization Principle we prove in Theorem \ref{Cr}. We decided to show such a result in the general setting of initial modules with respect to weights. To this goal, we propose a version of the Buchberger algorithm for weight orders (see Algorithm \ref{alg}), based on a strengthening of a lifting lemma proved by Buchsbaum and Eisenbud in \cite{BE} (see Lemma \ref{LIFT}). The technical efforts to get such a generalization are more serious than one might, at first blush, expect: while the computation of s-pairs is substituted by a syzygetic argument, the major difficulty when working with weight orders is the lack of a division algorithm.

Finally we show, in Theorem \ref{leila} that, under the assumption of componentwise linearity of the initial object, the equality between the $i$th Betti numbers of the module and its Gr\"obner deformation forces the componentwise linearity of the $j$th syzygy module, for all $j\geq i-2.$ For the proof, even when the module resolved is just an ideal, we need all the general theory we previously developed for modules and weights. 
\vspace{.3cm}\\
The first author would like to thank L. Sharifan for suggesting the statement of Theorem \ref{leila}, and M. Kummini for many discussions regarding Section 3 and especially for the reference to the work in \cite{BE}.


\section{Initial modules}

In this section we recall some basic facts about initial modules with respect to weight orders. For more detailed proofs we refer the reader to \cite[Ch. 8.3]{MS} and \cite[Ch. 15] {E}. We will adopt the same notation as in \cite{MS}.\\

Let $A=K[X_1,\dots,X_n]$ be a polynomial ring, let $\omega=(w_1,\dots,w_n)\in \mathbf Z^n_{\geq 0}$ be a non-negative integral weight, and let $F=\bigoplus_{j=1}^m A(-d_j)= \bigoplus_{j=1}^m A e_j$ a free graded $A$-module with basis $e_1,\dots,e_m$ of degrees $d_1,\dots,d_m.$ We let $\epsilon=(\epsilon_1,\dots, \epsilon_m)$ be a vector consisting of non-negative integers so that for every $j$, $\epsilon_j$ is the weight of $e_j.$ The weight of a monomial $\mathbf X ^{\mathbf u}e_j\in F$ is set to be $\omega \cdot \mathbf u + \epsilon_j.$\\

Given a non-zero element $f\in F$  we define its initial part with respect to $(\omega,\epsilon)$ in an obvious way: we first write $f$ uniquely as a sum with non-zero coefficients of monomials in $F$ and then we define $\ini_{(\omega,\epsilon)}(f)$ to be the sum with coefficients of all monomials with largest weight. 
Given a submodule $M$ of $F$, its initial module  $\ini_{(\omega,\epsilon)}(M)$  is just the submodule of $F$ generated by all $\ini_{(\omega,\epsilon)}(f)$ for $f\in M.$
When $M$ is graded, the modules $\ini_{(\omega,\epsilon)}(M)$ and $M$ have the same Hilbert function.\\

A standard method to work with initial modules is to use $(\omega, \epsilon)$ to construct a flat family in the following way. Let  $\tilde A$ be the  $\mathbb Z \times \mathbb Z$-graded polynomial ring $A[t]$ with $\deg(X_i)=(1,w_i)$ for all $i$, and  $\deg(t)=(0,1).$ Similarly, given $F= \bigoplus_{j=1}^m A e_j$, we define $\tilde F$ to be $\bigoplus_{j=1}^m \tilde A e_j$ with degree of $e_j$ equal to $(d_j, \epsilon_j).$

Given an element $f= \sum c_{\mathbf u,j} \mathbf X ^{\mathbf u} e_j \in F$ we let its homogenization in $\tilde F$ to be \[\tilde f= t^d \sum c_{\mathbf u,j} ( t^{-\omega \cdot \mathbf u} \mathbf X ^{\mathbf u}) (t^{-\epsilon_j} e_j),\] where $d$ is the largest weight of a monomial in the support of $f.$

Given a submodule $M$ of $F$ we define $\tilde M$ be the  submodule generated by all $\tilde f$, for $f\in M$. For an element $\alpha \in K$ we let $\tilde M_{t=\alpha}$ to be the image in $F$  of the module $\tilde M$ under the map evaluating $t$ at $\alpha,$ i.e. $\tilde M_{t=\alpha} \cong  \tilde M \otimes \tilde A/( t-\alpha).$ 

We have the following standard collection of facts (see \cite[Prop. 8.26, Prop. 8.28]{MS}).
\begin{proposition} \label{carrellata} Let $M\subseteq F$ be a graded submodule. Then 
\begin{enumerate}
\item $\tilde F /\tilde M$ is $K[t]$-free.
\item $\tilde M_{t=1}=M$ and $\tilde M_{t=0}= \ini_{(\omega, \epsilon)}(M),$ so $M$ and $\ini_{(\omega, \epsilon)}(M)$  have the same Hilbert function.

\item Given $\tilde{\mathbb F},$ 
a minimal graded $\tilde A$-free resolution of $\tilde F/ \tilde M$, one has that $\tilde{\mathbb F}\otimes \tilde A/(t-1)$ and $\tilde{\mathbb F}\otimes \tilde A/(t)$ are graded $A$-free resolutions of $F/M$ and $F/\ini_{(\omega, \epsilon)}(M).$ Furthermore $\tilde{\mathbb F}\otimes \tilde A/(t)$ is minimal, hence $\beta_{ij}(F/M)\leq \beta_{ij}(F/ \ini_{(\omega, \epsilon)}(M))$ for all $i,j.$
\end{enumerate}
\end{proposition}
\begin{remark}\label{INSIZ}
Consider, as in the above proposition, the resolutions $\tilde{\mathbb F}$, $\mathbb F:=\tilde{\mathbb F}\otimes \tilde A/(t-1)$ and  $\mathbb G:=\tilde{\mathbb F}\otimes \tilde A/(t)$ of $\tilde F/ \tilde M,$  $F/M$ and $F/\ini_{(\omega, \epsilon)}(M)$ respectively. Write $\tilde F_i=\bigoplus_{j=1}^{b_i}  \tilde A e_{ij}$, where $e_{ij}$ has bi-degree, say $(a_{ij},\epsilon_{ij}).$ Let $\epsilon(i)$ be the weight $(\epsilon_{i1}, \dots,\epsilon_{ib_i})$ and denote with $\tilde B$, $\tilde B _{t=1}$ and $\tilde B_{t=0}$ the images in $\tilde F_i,$  $F_i,$ and $G_i$ of the respective differential maps. Notice that, up to a graded isomorphism, we can assume that $F_i$ and $G_i$ are equal. \vspace{.3cm}\\
We claim that $\ini_{\omega,\epsilon(i)}(\tilde B_{t=1})= \tilde B_{t=0}.$ 
\vspace{.2cm}\\
First note that, by restricting $\mathbb F$ and $\mathbb G$ to all the homological degrees strictly greater than $i$, we obtain two graded resolutions with the same free modules. Thus  $\tilde B_{t=1}$ and $\tilde B_{t=0}$ have the same Hilbert function.
It is also a standard observation that for every $\alpha\in K$ the complex $\tilde{\mathbb F}\otimes \tilde A/(t-\alpha)$ is a graded $A$-free resolution of $\tilde F/\tilde M \otimes \tilde A/(t-\alpha),$ so by defining  
 $\tilde B_{t=\alpha}$ in an obvious way, we deduce that the the Hilbert function of $\tilde B _{t=\alpha}$ does not depend on $\alpha$ (this is true even when working with a field extension of $K$), and thus for every $i,$ the module $\tilde F_i/\tilde B$ is $K[t]$-free.
We deduce that both $t$ and $t-1$ do not divide the homogeneous minimal generators of $\tilde B$ hence, 
by setting $\tilde N \subseteq \tilde F_i$ to be the homogenization of $\tilde B_{{t=1}}$ with respect to $(\omega, \epsilon(i))$, we notice that  $\tilde N_{t=0}\subseteq \tilde B_{t=0}.$ The conclusion now follows because the left  hand side is precisely $\ini_{\omega,\epsilon(i)}(\tilde B_{t=1})$ and the equality between Hilbert functions forces the equality between the two modules.

\end{remark}
\section{A lifting Lemma} 
In this section we discuss a lifting lemma for  modules which is a strengthening of an analogous result of Buchsbaum and Eisenbud in \cite{BE}. Let $S$ be a commutative ring and let $l$ be a regular element of $S$. Set $R$ to be $S/(l).$ 
\begin{definition}
Given an $R$-module $M$ and an $S$-module $N$, we say that $(S,N)$ is a \emph{lifting} of $(R,M)$ (or simply that $N$ is a lifting of $M$) if $l$ is regular on $N$ and $N/lN\cong M.$
\end{definition}

The following is an adaptation of Lemma 2, Section 3 of \cite{BE}. It has a weaker assumption then the lemma in \cite{BE} (see the discussion following the statement) and moreover it includes not only a sufficient but also a necessary condition.

\begin{lemma}[Lifting]\label{LIFT} Let $R$ and $S$ as above. Assume that there exists the following commutative diagram of $S$-modules, with exact columns and exact bottom row (the two top rows may not be complexes):\\

\[
\begin{CD}
0                @.            0                @.           0             \\ 
@VVV                           @VVV                          @VVV          \\ 
G_2              @>g_2>>       G_1              @>g_1>>      G_0           \\
@VV \cdot lV                   @VV \cdot l V                 @VV\cdot lV   \\ 
G_2              @>g_2>>       G_1              @>g_1>>      G_0           \\
@VVp_2V                        @VVp_1V                       @VVp_0V       \\ 
F_2              @>f_2>>       F_1              @>f_1>>      F_0           \\
@VVV                           @VVV                          @VVV          \\ 
0                @.            0                @.           0           .             
\end{CD}
\]
Then ${\rm coker}(g_1)$ is a lifting (via $l$) of ${\rm coker}(f_1)$ if and only if 
$g_1 \circ g_2(G_2)\subseteq l \cdot g_1(G_1)$.\\
\end{lemma}

\begin{remark}
Notice that when the two top rows of the above diagram are complexes the condition 
$g_1 \circ g_2(G_2)\subseteq l \cdot g_1(G_1)$ is trivially satisfied. Lemma 2 of \cite{BE} is precisely the left  implication of Lemma \ref{LIFT} under this extra assumption.
\end{remark}
\begin{proof}
Assume that  ${\rm coker}(g_1)$ is a lifting  of ${\rm coker}(f_1);$ hence $G_0/(g_1(G_1)+lG_0)\cong  F_0/f_1(F_1),$ where the isomorphism is the one induced by the projection $p_0.$ Let $a_2$ be an element of $G_2.$ From the commutativity of the diagram we have that  $p_0(g_1(g_2(a_2)))=0,$ thus  $g_1(g_2(a_2))=l b_0$ with $b_0\in G_0.$ Since $l$ is a  regular on $G_0/g_1(G_1),$ we deduce that $b_0\in  g_1(G_1).$

Assume now that  $g_1 \circ g_2(G_2)\subseteq l \cdot g_1(G_1).$ It is enough to show that $l$ is regular on $G_0/g_1(G_1).$  Let $a_0\in G_0$ such that $la_0=g_1(a_1)$ for some $a_1\in G_1.$ Since $p_0(la_o)=0,$ from the commutativity of the diagram we get that $f_1(p_1(a_1))=0,$ in particular $p_1(a_1)=f_2(b_2)$ for some $b_2\in F_2.$ Let $a_2$ be an element of $G_2$ with $p_2(a_2)=b_2.$ Consider $(g_2(a_2)-a_1)$ and apply $p_1.$ We get $p_1(g_2(a_2))-p_1(a_1)$ which is equal to $f_2(p_2(a_2))-f_2(b_2)$ and therefore it is zero. Hence $g_2(a_2)-a_1$ is in the kernel of $p_1.$ We can write $g_2(a_2)-a_1=l\tilde{a_1}$ for some $\tilde{a_1} \in G_1.$ By applying $g_1$ we get $lg_1(\tilde{a_1})=g_1(g_2(a_2))-g_1(a_1)=lg_1(a_1') -l a_0$ for some $a_1'\in G_1$ which exists by assumption. Since $l$ is a nonzero divisor on $G_0$ we deduce that $a_0=g_1(a_1')-g_1(\tilde{a_1})\in g_1(G_1).$ 
\end{proof}
\section{Buchberger criterion and algorithm for weight orders}

We adopt the same notation as in Section 2.
The following simple fact says that, essentially, the computation of a Gr\"obner basis corresponds to a saturation. 
\begin{proposition} \label{SAT} Let $M=\lr{f_1,\dots,f_r}\subseteq F$ then $\tilde{M}=\lr{\tilde{f_1},\dots,\tilde{f_r}}:t^{\infty}.$
\end{proposition}
\begin{proof} Since $\tilde{F}/\tilde{M}$ is free as $K[t]$-module and $\lr{\tilde{f_1},\dots,\tilde{f_r}}:t^{\infty}$
is the smallest submodule of $\tilde F$ containing $\lr{\tilde {f_1},\dots,\tilde{f_r}}$ such the quotient of $\tilde{F}$ by such module is $K[t]$-free, we get $\tilde{M}\supseteq \lr{\tilde{f_1},\dots,\tilde{f_r}}:t^{\infty}.$ On the other hand by adding $(t-1)\tilde F$ to both sides we get $\tilde{M} +(t-1)\tilde F \supseteq \lr{\tilde{f_1},\dots,\tilde{f_r}}:t^{\infty}+ (t-1)\tilde F \supseteq M+(t-1) \tilde F,$ where all containments are forced to be equalities. The conclusion follows because $\tilde F/ \tilde M$ and  $\tilde F / \lr{\lr{\tilde{f_1},\dots,\tilde{f_r}}:t^{\infty}}$ are $K[t]$-free. 
\end{proof}

The following Algorithm \ref{alg} is an analogue, when working with weight orders, of  the famous Buchberger's algorithm. The usual computation of s-pairs is replaced by a calculation done using  the generators of the syzygy module of  certain initial elements (which may not be monomial). The computation of standard expressions, which is not possible in the generality considered here because of the lack of a division algorithm, is avoided thanks to the lifting lemma proved in the previous section.

It is important to stress that the scope of our algorithm is more theoretical than computational. 
Our goal, within this paper,  is to derive the two corollaries which conclude the section. 
We disregard completely both the computational problem of finding generators for the syzygy modules, and also the verification of the membership of an element in a given submodule.

Given an element $q \in \tilde F$ we denote with $\underline q\in F$ its image under the map evaluating $t$ at zero.
Analogously $\underline Q$ will denote the image in $F$ of the submodule $Q$ of $\tilde F.$

\begin{algorithm}\label{alg} Let $M=\lr{f_1,\dots,f_r}\subseteq F,$ and denote $\lr{\tilde{f_1},\dots, \tilde{f_r}}$ by $M_0.$ 
We have an exact sequence derived from the free resolution of $F/\underline {M_0}$
\[A^{s}\stackrel{ \Phi}{\rightarrow} A^{r} \stackrel{(\underline{\tilde{f_1}},\dots,\underline{\tilde{f_r}})} \rightarrow F.
\] where $\Phi$ is a matrix whose entries are polynomials of $A.$
Consider
 \[
\tilde{A}^{s}\stackrel{\Phi}{\rightarrow} \tilde{A}^{r} 
\stackrel{(\tilde{f_1},\dots,\tilde{f_r})} \rightarrow \tilde F,
\] and notice that this may not even be a complex. Notice also that $(\tilde{f_1},\dots,\tilde{f_r})\circ \Phi(\tilde A^s)$ is equal to $t Q$ for some bihomegeneous submodule $Q$ of $\tilde{F}.$ Apply Lemma \ref{LIFT}, with $l=t$,  to deduce that either $t$ is regular on $\tilde{F}/M_0$ or that $Q \not \subseteq M_0,$ in this latter case we define $M_1:=M_0+Q.$
Now we replace $M_0$ by $M_1$ and repeat the same argument. 

In this way, we construct an increasing  chain $M_0\subset M_1 \subset \cdots $  of submodules of $\tilde F$ that, by the Noetherian property, must terminate, say at $M_j$. Since the chain cannot become longer, $\tilde{F}/M_j$ is therefore $K[t]$-free. Thus $\lr{\tilde f_1,\dots,\tilde f_r} :t^{\infty} \subseteq M_j.$ Each module in the chain is bihomogeneous and is mapped surjectively to $M$ under the evaluation of $t$ at $1.$ Therefore $M_j$ is contained in the homogenization $\tilde{M}$ of $M.$ Hence, by using Proposition \ref{SAT}, we get that $M_j=\tilde M.$ Thus $(M_j)_{t=0}= \ini_{(\omega,\epsilon)}(M).$ 
\end{algorithm}
From the above algorithm we derive immediately the following corollary.
\begin{corollary}[Buchberger's criterion for weight orders] \label{BU} Let $M\subseteq F$ be an $A$-module generated
by $f_1,\dots, f_r.$  Consider $A^{s}\stackrel{ \Phi}{\rightarrow} A^{r}
\rightarrow F,$ a presentation of $\lr{\ini_{(\omega,\epsilon)} {f_1},\dots, \ini_{(\omega,\epsilon)}{f_r}}.$ 
If $\frac{1}{t}(\tilde{f_1},\dots,\tilde{f_r}) \circ \Phi(\tilde A ^s) \subseteq \lr{\tilde f_1,\dots, \tilde f_r}$ then $\ini_{(\omega,\epsilon)}(M)=\lr{\ini_{(\omega,\epsilon)} {f_1},\dots, \ini_{(\omega,\epsilon)}{f_r}}.$ 
\end{corollary}
The following result will play a key role in the proofs discussed in Section 5.
\begin{corollary}\label{SH} Let $M\subseteq F$ be a graded $A$-module generated
by homogeneous elements $f_1,\dots, f_r.$ Consider a graded presentation $A^{s}\stackrel{ \Phi}{\rightarrow} A^{r} \rightarrow F$ of $\lr{\ini_{(\omega,\epsilon)} {f_1},\dots, \ini_{(\omega,\epsilon)}{f_r}}$  and assume that $\Phi(A^{s})$ is generated in degree at most $d.$ 
Then $in_{(\omega,\epsilon)}(M)$ and $\lr{\ini_{(\omega,\epsilon)} {f_1},\dots, \ini_{(\omega,\epsilon)}{f_r}}$ are equal provided the two modules agree in all degrees less than or equal to $d.$
\end{corollary}
\begin{proof} We follow the same notation of Algorithm \ref{alg} and let $M_0\subseteq \tilde F$ be $\lr{\tilde f_1,\dots, \tilde f_r}.$ By assumption,  the generating degree of $(\tilde{f_1},\dots,\tilde{f_r})\circ \Phi(\tilde A^s)=t Q$ has first component bounded by $d.$ We have $(M_0)_{t=0}\subseteq (M_0+Q)_{t=0} \subseteq \tilde M_{t=0}.$ Since the Hilbert functions of the first and of the last module agree in all degrees less than or equal to $d,$ we deduce that the there modules agree in all such degrees. On the other hand $\tilde F/ \tilde M$ is $K[t]$-free, so from the inclusion $M_0\subseteq (M_0+Q)\subseteq \tilde M$ we deduce can 
that $M_0$, $M_0+Q$  and $\tilde M$ agree in all multidegrees whose first component is bounded by $d,$ provided we can show that in these degrees $\tilde F/ (M_0)$ and  $\tilde F/(M_0+Q)$ are $K[t]$-free (i.e. have no torsion 
 as $K[t]$- modules). Fix a degree $a\leq d$ and consider $C=(\tilde F/ (M_0))_{(a,\bullet)}$ consisting of all the bigraded components whose degree has first entry equal to $a.$ Notice that $C$ is a f.g. graded $K[t]$-module, hence we can apply the structure theorem for modules over PID (in the graded setting) and deduce that the torsion of $C$, if any, if of the form $\oplus_i  K[t]/(t^{a_i}).$ On the other hand if this torsion is not zero $(\tilde F/M_0) \otimes \tilde A/ (t)= F/\lr{\ini_{(\omega,\epsilon)} {f_1},\dots, \ini_{(\omega,\epsilon)}{f_r}}$   and  $(\tilde F/M_0) \otimes \tilde A/ (t-1) = F/ M$  would have Hilbert functions with different values at $a,$ which is impossible since those Hilbert functions agree, at $a$,  with the one of  $F/\ini_{(\omega, \epsilon)}(M).$
An analogous argument works for $\tilde F/(M_0+Q).$ Hence $M_0=M_0+Q,$ and in particular, since we now know that $Q\subseteq M_0$, the algorithm terminates at  $M_0.$
\end{proof}

\section{Componentwise regularity}\label{sec-creg}
In this section we prove some properties regarding the $m$-componentwise regularity and the componentwise linearity.

Recall that, by \cite{AE,AP}, a finitely generated standard graded $K$-algebra $R$ is Koszul if and only if every f. g. graded $R$-module $M$ has finite Castelnuovo-Mumford regularity. 

Let  $R$ be Koszul and let $M$ be a f. g.  graded $R$-module generated in degree less then or equal to $d.$ It is not hard to see that $\reg(M)\leq d+\creg(M).$
 Furthermore, from the definition and from the fact that over a Koszul algebra $\reg(\mathbf mM)\leq \reg(M)+1$, one has that $\creg(M)= \max\{\reg (M_{\lr{d}})-d : \beta_{0d}(M)\not =0\}.$ Hence, over a Koszul algebra, every f. g. graded module has finite componentwise regularity. 
 On the other hand, given a f. g. standard graded $K$-algebra $R$ such that every f. g. graded $R$-module has finite componentwise regularity, one has $\reg(R/ \mathbf m)= \creg(R/ \mathbf m)<\infty$. This is enough, by \cite{AP}, to conclude that $R$-is Koszul. We have thus proven that $R$ is Koszul if and only if every f. g. graded $R$-module has finite componentwise regularity.

In general, one cannot expect to obtain good upper bounds for the componentwise regularity in terms of the Castlnuovo-Mumford regularity. For instance by considering an ideal of $I\subseteq A=K[X_1,\dots,X_n]$ generated in a single degree $d$ and such that $d  \ll \reg(I),$ we see that $\reg( I+ \mathbf m^{d+1})=d+1  \ll \creg(I+\mathbf m^{d+1})=\reg(I)-d.$ Hence, even when working with homogeneous ideals  in a polynomial ring, a general bound for the componentwise regularity in terms of the regularity is as bad as the bound for the regularity in terms of the generating degree, i.e. double exponential (see \cite{CS2}).

Given a finitely generated graded $R$-module $M$ we denote by $\Omega_i^R(M)$ the $i$th syzygy module of $M$, that is, the kernel of the map $\partial_i:F_i\rightarrow F_{i-1}$, where $\{F_{\bullet} ,\partial_{\bullet}\}$ is a minimal graded free resolution of $M$ over $R$ (here $F_{i-1}=M$ and $\partial_0$ is the presentation $F_0\rightarrow M$). Whenever $i<0$ we set $\Omega_i^R(M)$ to be equal to $M.$
 The maps $\partial_i$ are represented by matrices whose entries are homogeneous elements of $R$. By putting at $0$ all the entries of degree $>1$, we get a new complex, $\lin(F_{\bullet})$, and we can define the {\it linear defect of $M$ as}:
\[\ld_R(M)=\sup\{i:H_i(\lin(F_{\bullet}))\neq 0\}.\]
In his PhD thesis \cite[Theorem 3.2.8]{R}, R\"omer proved that, provided $R$ is Koszul, the componentwise linearity of $M$ is characterized by the fact that $\ld_R(M)=0$. This implies at once the following, a priori not clear, fact:

\begin{lemma}\label{RO}
Let $M$ be a f. g. graded module over a Koszul algebra $R$. If $\Omega_i^R(M)$ is componentwise linear, then $\Omega_j^R(M)$ is componentwise linear for all $j\geq i$.
\end{lemma}

For further properties of the linear defect of modules see \cite{IR}.
We adopt the same notation as Section 2 and Section 4.

\begin{theorem}[Crystallization]\label{Cr}  Let $M\subseteq F$ be a f. g. graded $A$-module generated in degrees less than or equal to $a$. If $\ini_{(\omega,\epsilon)}(M)$ has no minimal generators in degrees $a+1,\ldots,a+1+\creg(\ini_{(\omega,\epsilon)}(M))$, then $\ini_{(\omega,\epsilon)}(M)$ is generated in degree less than or equal to $a$.
\end{theorem}
\begin{proof} 
Since the initial modules of $M$ and $M_{\lr{a}}$ agree in all the graded components greater than or equal to $a,$ we can assume without loss of generality that $M=M_{\lr{a}}.$  Let $f_1,\dots,f_r$ be the homogeneous minimal generators of $M.$ By definition $\reg(\lr{\ini_{\omega,\epsilon)}f_1,\dots,   \ini_{(\omega,\epsilon)}f_r})\leq a+ \creg(\ini_{(\omega,\epsilon)}(M)).$ By Corollary \ref{SH} we deduce that 
$\ini_{(\omega,\epsilon)}(M)=\lr{\ini_{(\omega,\epsilon)}f_1,\dots, \ini_{(\omega,\epsilon)}f_r}.$  
\end{proof}

\begin{remark} It is useful to notice that the above result is still valid provided the notion of componentwise regularity is replaced by a weaker one defined in terms of the syzygy module. Precisely we can replace $\creg(\ini_{(\omega,\epsilon)}(M))$ by an integer $r$ such that $\beta_{1,d+1}(\ini_{(\omega,\epsilon)}(M))_{\lr{a}}=0$ for all $d>a+r.$ 
\end{remark}

The following result for graded modules, which is quite a standard fact when working with homogeneous ideals and term orders, is a direct consequence of Theorem \ref{Cr}.

\begin{theorem}\label{SameB} Let $M\subseteq F$ a f. g. graded $A$-module and assume that  $\ini_{(\omega,\epsilon)}(M)$ is componentwise linear. If  $\beta_0(M)= \beta_0( in_{(\omega,\epsilon)}(M))$ then $M$ is componentwise linear as well. 
\end{theorem}
\begin{proof} 
The hypothesis $\beta_0(M)= \beta_0(\ini_{(\omega,\epsilon)}(M))$, together with Proposition \ref{carrellata} (3), imply for all degree $d$ that $\beta_{0d}(M)= 
\beta_{0d}(in_{(\omega,\epsilon)}(M)).$ 
Since the Hilbert functions of $M$ and $\ini_{(\omega,\epsilon)}(M)$ are the same, the equality between $0$-graded Betti numbers forces $\mathbf m M$ and $\mathbf m  \ini_{(\omega,\epsilon)}(M)$ to have the same Hilbert function.
From the equality between the $0$-graded Betti numbers and the equality of Hilbert functions one deduces that, for every degree $a$, the initial module of $M_{\lr a}$ has no generator in degree $a+1$ (otherwise the extra miminal generators would contribute to the minimal generators of the initial module of $M$). By Theorem \ref{Cr} the initial module of $M_{\lr a}$ (if not zero) is generated in degree $a,$ and thus agrees with $(\ini_{(\omega, \epsilon)}(M))_{\lr a}.$ By Proposition \ref{carrellata} (3), $\reg M_{\lr a} \leq \reg  \ini_{(\omega, \epsilon)}(M _{\lr a}),$ and by assumption the latter is bounded by $a.$
\end{proof}

\begin{remark}
Under certain assumptions, the implication of Theorem \ref{SameB} can be reversed. For instance when the characteristic of $K$ is zero and $\ini_{(\omega,\epsilon)}(M)$ is replaced by the generic initial module with respect to a reverse lexicographic order on $F$ (see [NR] for the ideal theoretic case).
\end{remark}

By using the fact that the graded Betti numbers of $M$ can be obtained from those of  $\ini_{(\omega,\epsilon)}(M)$  by consecutive cancellations (see \cite{P}), we see that if  $\beta_1(M)=\beta_1(\ini_{(\omega,\epsilon)}(M))$ then $\beta_0(M)$ and $\beta_0(\ini_{(\omega,\epsilon)}(M))$ are equal as well.  Hence, from Theorem \ref{SameB}, we derive the following:

\begin{corollary} \label{SameB2} Let $M\subseteq F$ a f. g. graded $A$-module and assume that  $\ini_{(\omega,\epsilon)}(M)$ is componentwise linear. If  $\beta_1(M)= \beta_1( in_{(\omega,\epsilon)}(M))$ then $M$ is componentwise linear as well.
\end{corollary}
\begin{theorem}\label{leila} Let $M\subseteq F$ be a f. g. graded $A$-module and assume that  $\ini_{(\omega,\epsilon)}(M)$ is componentwise linear. If $\beta_i(M)=\beta_i(\ini_{(\omega,\epsilon)}(M))$ for some $i\geq 0,$ then $\Omega_{j}^R(M)$ is componentwise linear for all $j\geq i-2.$ 
\end{theorem}
\begin{proof}  Because of Theorem \ref{SameB}, Corollary \ref{SameB2} and Lemma \ref{RO} we can assume, without loss of generality, that $i\geq 2.$
Notice that for $j$ and $b$ non-negative $\Omega_{j+b+1}^R(M) \cong \Omega_{j}^R(\Omega_{b}^R(M))$, hence by Lemma \ref{RO} it is enough to prove that $\Omega_{i-2}^R(M)$ is componentwise linear. We adopt a notation analogous to the one of  Remark \ref{INSIZ} and consider the (not necessarily minimal) graded free resolution 
$\mathbb F:=\tilde{\mathbb F} \otimes \tilde A/(t-1)$ of  $F/M,$ and  the minimal graded free resolution 
$\mathbb G:=\tilde{\mathbb F} \otimes \tilde A/(t)$ of  $F/\ini_{(\omega,\epsilon)}(M).$ 
We denote with $\tilde B$, $\tilde B _{t=1}$ and $\tilde B_{t=0}$ the images in $\tilde F_{i-2},$  $F_{i-2},$ and $G_{i-2}$ of the respective differential maps. By Remark \ref{INSIZ} we know that $\ini_{\omega,\epsilon(i-2)}(\tilde B_{t=1})= \tilde B_{t=0}.$ By Lemma \ref{RO}, since  $\ini_{(\omega,\epsilon)}(M)$ is componentwise linear, we deduce that its $(i-2)$th syzygy module, i.e. $\tilde B_{t=0}$, is componentwise linear as well. 
We have the following inequalities between number of generators  $\beta_i(M)= \beta_1(\Omega_{i-2}^R(M))\leq  \beta_1(\tilde B_{t=1}) \leq \beta_1(\tilde B_{t=0}) =
\beta_i(\ini_{(\omega,\epsilon)}(M))$, thus by Corollary \ref{SameB2},
$\tilde B_{t=1}$ is componentwise linear.
As a consequence of Schanuel's lemma (see \cite[Thm. 20.2]{E}), since  $\mathbb F$
may not be minimal,  $\tilde B_{t=1}$ and  $\Omega_{i-2}^R(M)$ are isomorphic up to a graded free summand, hence 
$\Omega_{i-2}^R(M)$ is componentwise linear as well.
\end{proof}
\begin{remark} When $M=I\subseteq A$ is a homogeneous ideal, under the same assumption as the above theorem, it was proven in \cite{CHH} that $\beta_j(I)=\beta_j(\ini(I))$ for all $j\geq i.$ It is possible to show that the same rigidity is also true when working with graded modules; unfortunately the technical argument required goes behind the scope of the present paper. 
\end{remark}



\end{document}